\documentclass[10pt]{amsart}

\usepackage{amsmath,amstext,amscd, amssymb,txfonts, epsfig, psfrag, color, epstopdf}
\usepackage[normalem]{ulem}
\usepackage[colorlinks=true, pdfstartview=FitV, linkcolor=blue, citecolor=blue, urlcolor=blue]{hyperref}
\usepackage{algorithm}
\usepackage{algorithmicx,algpseudocode}

\definecolor{cerulean}{rgb}{0,.48,.65} 
\definecolor{magenta}{rgb}{.5,0,.5} 
\definecolor{dred}{rgb}{.5,0,0} 
\definecolor{green}{rgb}{0,.5,0} 
\definecolor{blue}{rgb}{0,0,0.5} 
\definecolor{black}{rgb}{0,0,0} 
\definecolor{dgreen}{rgb}{0,.3,0} 
\definecolor{vdred}{rgb}{.3,0,0} 
\definecolor{red}{rgb}{1,0,0} 
\definecolor{salmon}{rgb}{0.98,0.50,0.45} 
\definecolor{gray}{rgb}{.5,.5,.5} 
\definecolor{seagreen}{rgb}{0.13,0.70,0.67} 
\definecolor{chartreuse}{rgb}{0.40,0.80,0.00}
\definecolor{cornflower}{rgb}{0.39,0.58,0.93} 
\definecolor{gold}{rgb}{0.80,0.68,0.00}
 
\theoremstyle{plain}

\newtheorem{theorem}{Theorem}
\newtheorem{thm}{Theorem}[section]
\newtheorem{lemma}[thm]{Lemma}

\theoremstyle{definition}

\newtheorem{Open questions}[thm]{Open questions}
\newtheorem{Open question}[thm]{Open question}
\newtheorem{Open problems}[thm]{Open problems}
\newtheorem{Open problem}[thm]{Open problem}

\newlength{\foo}
\newlength{\bah}
\def\Bbb{\mathbb}

\def\Z{\Bbb{Z}}
\def\N{\Bbb{N}}

\def\ni{\noindent}

\def\Area{\hbox{\rm Area}}

\def\R{\hbox{\rm R}}

\def\F+L{\hbox{$\textup{F}\!_+\textup{L}$}}

\def\ssm{\smallsetminus}

\def\onto{{\kern3pt\to\kern-8pt\to\kern3pt}}

\def\<{\langle}
\def\>{\rangle}
\def\|{{\ |\ }}

\newcommand{\set}[1]{\left\{#1\right\}}
\newcommand{\restricted}[1]{\left|_{#1} \right.}
\newcommand{\abs}[1]{\left|#1\right|}
\renewcommand{\ni}{\noindent}
\renewcommand{\ss}{\smallskip}

\def\*{^{\star}}

\begin{document}

\title{Computing area in presentations of the trivial group}

\author{Timothy Riley}

\date \today

\begin{abstract}
\ni    We give polynomial-time  dynamic-programming algorithms finding the areas of words in the   presentations  $\langle a, b \mid a, b \rangle$ and    $\langle  a, b \mid a^k, b^k; \ k \in \N \rangle$ of the trivial group. 

In the first of these two cases, area  was studied under the name \emph{spelling length} by Majumdar, Robbins and Zyskin in the context of the design of liquid crystals.  We explain how the problem of calculating it can be  reinterpreted in terms of RNA-folding.   In the second,  area is what   Jiang called \emph{width} and studied when counting fixed points for  self-maps of a compact surface, considered up to homotopy.    
In 1991 Grigorchuk and Kurchanov gave an algorithm computing width and asked whether it could be improved to polynomial time.  We  answer this affirmatively.    
 \ss
\\
\ni \footnotesize{\textbf{2000 Mathematics Subject Classification:  20F05, 20F10, 68W32}}  \\ 
\ni \footnotesize{\emph{Key words and phrases:} exact area, group presentation, width, RNA folding}
\end{abstract}

\maketitle

\section{Introduction} \label{intro}

A word  $w$ on an alphabet $A^{\pm 1}$ represents the identity in the group    presented by  $\langle A \mid R \rangle$ when  $w$ freely equals $(u_1^{-1} r_1^{\epsilon_1} u_1) \cdots (u_N^{-1} r_N^{\epsilon_N} u_N)$ for some words $u_i$ on $A^{\pm 1}$, some $\epsilon_i \in \set{1, -1}$,  some $r_i \in R$, and some integer $N \geq 0$.   Denote by $\Area(w)$ the  minimal $N$ among all such products.  

The reason    \emph{area} is an appropriate term here is that   such a product  has a natural geometric interpretation as a disc (more precisely, a `van~Kampen diagram') of area $N$ spanning a loop associated to $w$.    There is an extensive literature on   optimal upper bounds (`Dehn functions') on $\Area(w)$ as a function of the length $\ell(w)$ of $w$.   Such bounds are usually presented asymptotically and   involve considering only `worst' instances of $w$ within a given length. Our focus here, by contrast, is  calculating $\Area(w)$ \emph{exactly} for \emph{all} $w$ that represent the identity.  Perhaps surprisingly, this turns out to be subtle even for some of the most elementary   presentations of the trivial group.   

Here is the main theorem we will discuss in this paper. 
(For simplicity, we work with an alphabet of two letters $a$ and $b$, but all our results and arguments easily extend to any finite alphabet.)  

\begin{theorem} \label{main}
There are deterministic algorithms to compute $\Area(w)$ in 
\begin{enumerate}
\item   $\langle a, b \mid a, b \rangle$ in time $\tilde{O}(n^{2.8603})$, \label{first}
\item  $\langle  a, b \mid a^k, b^k; k \in \N \rangle$  in time $\tilde{O}(n^4)$, \label{second}
\end{enumerate}
where $n = \ell(w)$.  
\end{theorem}
The $\tilde{O}$-notation here differs from $O$-notation in that it hides  a polylogarithmic factor.  We obtain $\tilde{O}$-estimates by   counting arithmetic operations and look-ups from tables; the notation allows us  to suppress the cost of performing the arithmetic.    
 
Part (2) of Theorem~\ref{main} answers a question of Grigorchuk and Kurchanov as I will discuss further in the next section.

Given that  Dehn functions are long-studied and  that $\langle a, b \mid a, b \rangle$ is one of the   most elementary presentations of the trivial group imaginable, it is hard to think there might be  much to computing area there.   Here is why it is not so easy.   A word $w$ on $a^{\pm 1}, b^{\pm 1}$ can be converted to the empty word by a sequence (called a \emph{null-sequence})  of two types of move: (1)  delete a letter, (2)  cancel an inverse pair of adjacent letters $a^{\pm 1} a^{\mp 1}$ or $b^{\pm 1} b^{\mp 1}$.  An equivalent definition of $\Area(w)$ in  $\langle a, b \mid a, b \rangle$  is the minimal $N$ such that there is a null-sequence  that employs $N$ moves of the first type.  (This equivalence is a special case of a well-known general relationship between null-sequences and area, explained for instance in Section~II.2.5  of \cite{BRS}.)   So the challenge is to use the moves of the first type expediently  to maximize the number of moves of the second type.  For instance $\Area(a^2b a^{-2} b^{-1}) =2$: the best one can do is delete the $b$ and the $b^{-1}$, and cancel the $a^2$ with the $a^{-2}$.       

Perhaps surprisingly, part (1) of Theorem~\ref{main} leads us to biology.  
A dynamic-programming algorithm by Ruth~Nussinov and  Ann~Jacobson  in their influential 1980 article~\cite{NJ} on RNA-folding computes area in $\langle a, b \mid a, b \rangle$  in cubic time.  The point is that $a$, $a^{-1}$, $b$ and $b^{-1}$  can be thought of like nucleotides;   $a$ with $a^{-1}$ and   $b$ with $b^{-1}$  correspond to  matched base pairs.  The problem Nussinov and Jacobson posed and   solved   is  to find a way for an RNA strand to fold against itself so that it maximizes the number of matched base pairs.    This folding is analogous to constructing a null-sequence because of  a ``non-crossing'' (or ``non-knotting'') condition on the matched pairs---the word cannot have the form $u_1x u_2 y u_3 x^{-1} u_4 y^{-1} u_5$ where  $x,x^{-1}$ and  $y,y^{-1}$ are two matched pairs and $u_1, \ldots, u_5$ are subwords. This condition vaguely corresponds to a constraint that the RNA strands should not form a knot in 3-space.    (There is an extensive  literature on this type of problem  in both bioinformatics and computer science.  The survey   \cite{Kao} is a good starting point.)    
 
Here is a  translation of Nussinov and  Jacobson's algorithm to our setting. 
On input a word $w= x_1x_2\cdots x_n$, where   $x_1, \ldots, x_n \in \set{a^{\pm 1}, b^{\pm 1}}$, the idea is to compute  an array of integers $A_{i,j}$ for $1\leq i \leq j \leq n$ which will equal $\Area(x_i \cdots x_j)$ in $\langle a, b \mid a, b \rangle$.  In particular, $A_{1,n}$ will be $\Area(w)$.
 
 \begin{algorithm}[ht]
    \caption{--- Area in $\langle a, b \mid a, b \rangle$ \newline 
  $\circ$ \  Input  a word $w= x_1x_2\cdots x_n$   where  $x_1, \ldots, x_n \in \set{a^{\pm 1}, b^{\pm 1}}$.  \newline  
  $\circ$ \  Return  $\Area(w)$ in time  $\tilde{O}(n^3)$.   }
    \label{Alg: first}
    \begin{algorithmic}
		\State Define $A_{i,i} := 1$ for $i =1, \ldots, n$  and    $A_{i,j} := 0$ for all $i>j$ 
		\State  For $k = 1$ to $n-1$  
		\State  \qquad For $i =1, \ldots, n-k$ define $A_{i,i+k}$ to be the minimum of  
	 	\State \qquad \ \ $\set{A_{i,i+k-1}+1} \cup \set{ A_{i,r-1} + A_{r+1,i+k-1}  \,  \left| \,  i \leq r < i+k \text{ and }  x_r = x_{i+k}^{-1} \right.}  $ 
		\State Return $A_{1,n}$ 
  \end{algorithmic}
  \end{algorithm} 
 
 The reason this algorithm is correct is that the optimal null-sequence for $x_i \cdots x_{i+k}$ either deletes $x_{i+k}$, or it pairs off $x_{i+k}$
with $x_r$ for some $i \leq r<i+k$ such that $x_r$ and $x_j$ are inverses
of each other.  In the first case $A_{i, i+k} = A_{i,i+k-1}+1$.  In the second $A_{i,i+k} = A_{i,r-1} + A_{r+1,i+k-1}$.    The algorithm halts in time $\tilde{O}(n^3)$ because, after computing $A_{i,j}$ for every $i,j$ such that $j-i \leq k-1$, it only takes an additional $O(k)$ operations to compute $A_{i,i+k}$.   

In the decades since Nussinov and   Jacobson's $\tilde{O}(n^3)$ bound, a number of authors have made improvements that shave off  log factors.  Recently Karl~Bringmann, Fabrizio~Grandoni, Barna~Saha, and Virginia~Vassilevska~Williams \cite{BGSW} broke the  $n^3$-barrier by combining fast-matrix multiplication methods and an algorithm of Leslie~Valiant for parsing context free grammars to give an algorithm which runs in time  $\tilde{O}(n^{2.8603})$, so this is the bound we give for \eqref{first} of Theorem~\ref{main}.

The problem  of calculating area in $\langle  a, b \mid a^k, b^k;  \ k \in \N \rangle$ in polynomial time  was posed by Grigorchuk and Kurchanov in their 1991 paper  \cite{GrK}.   The solution we will give, proving   \eqref{second} of Theorem~\ref{main}, will blend   Algorithm~\ref{Alg: first} with another famous dynamic-programming algorithm which we give  below as Algorithm~\ref{Alg: subset sum}.

We stress   that the bounds in Theorem~\ref{main} are in terms of $n=\ell(w)$.  The situation for   $w = a^{i_1} b^{i'_1} .... a^{i_k} b^{i'_k}$  inputted as a sequence of \emph{binary} integers $i_1, i'_1, ..., i_k, i'_k$ is markedly different.  A comparison  with   the subset sum problem, which asks, given   integers $i_1, \ldots, i_k$, whether there  are $j_1 < \cdots <  j_l$  with $l \geq 1$ such that $i_{j_1} + \cdots + i_{j_l} = 0$, makes this clear.  We will prove in Section~\ref{G-K}:

\begin{theorem} \label{comparison}
Computing the areas  of words in  $\langle  a, b \mid a^k, b^k; \ k \in \N \rangle$ is at least as hard as subset sum   in that for non-zero integers $i_1, \ldots, i_k$,   $\Area(a b^{i_1} a b^{i_2} \ldots a b^{i_k}) \leq k$ if and only if there are $j_1 < \cdots <  j_l$ (with $l \geq 1$)
such that $i_{j_1} + \cdots + i_{j_l} = 0$.
\end{theorem}

For $i_1, \ldots, i_k$ inputted in binary,   subset sum is NP-complete \cite{Karp}.  
 But here is a  well-known dynamic programming  algorithm  solving it deterministically in polynomial time as a function of $n := |i_1| + \cdots + |i_k|$.    The idea is to compute an array $S_{p,q}$ where $-n \leq p \leq n$ and $1 \leq q \leq k$ such that $S_{p,q} = 1$ when $i_{j_1} + \cdots + i_{j_l} = p$ for some  $q \leq j_1 < \cdots <  j_l \leq k$ (with $l \geq 1$) and  $S_{p,q} = 0$ otherwise.  

 \begin{algorithm}[ht]
    \caption{--- Subset sum \newline 
  $\circ$ \  Input  non-zero integers   $i_1, \ldots, i_k$.  Define  $n := |i_1| + \cdots + |i_k|$.  \newline  
  $\circ$ \  Declare   in time  $O(n^2)$ whether  $\exists$ $j_1 < \cdots <  j_l$ with $l \geq 1$ such that $i_{j_1} + \cdots + i_{j_l} = 0$.   }
    \label{Alg: subset sum}
    \begin{algorithmic}
		\State Define $S_{p,q} := 0$ for $p,q$ outside ($-n \leq p \leq n$ and $1 \leq q \leq k$)
		\State  For $q = k-1$ to $1$  
		\State \qquad For  $p = -n$ to $n$    
		\State  \qquad \qquad Define $S_{p,q} := \begin{cases}
		   1 \ \text{ if  } S_{p,q+1} =1 \text{ or }  S_{p-i_q,q+1} =1   \text{ or } i_q= p  \\  
	 	  0  \ \text{ otherwise } 
		\end{cases}  $ 
		\State Return $S_{0,1}$ 
  \end{algorithmic}
  \end{algorithm} 

This   works because when $i_{j_1} + \cdots + i_{j_l} = p$ for some  $q \leq j_1 < \cdots <  j_l \leq k$ (with $l \geq 1$), either   $i_q$ contributes to the sum (that is, $q  = j_1$) or it does not (that is, $q  < j_1$), and the first of these possibilities divides into two cases according to whether or not $i_q= p$.
Restricting the range of $p$ to $-n \leq p \leq n$ is appropriate because any sum of numbers from the list $i_1, \ldots, i_k$ has absolute value at most $n$.  The running time (the number of  look-ups from the prior completed parts of the array $S_{p,q}$ plus the number of arithmetic operations) is at most the size of the array, which is $O(n^2)$.     

Faster algorithms for subset sum have been found, including most recently \cite{KX}.

\section{Background, attributions, and acknowledgements} \label{background}
 
I first came to the topic of this article from the unlikely direction of liquid-crystal design.  I thank  Jonathan Robbins    for this.  He introduced me  to the problem of calculating area in  $\langle a, b \mid a, b \rangle$ in 2008, which he and his coauthors, Apala~Majumdar and Maxim~Zyskin,  called   \emph{spelling length}. They wished to calculate it because they had an application to  the design of `nematic liquid crystals in confined polyhedral geometries'        \cite{Robbins2,Robbins1}.   We conclude this article in Section~\ref{lcds} with a  sketch of how the connection  to combinatorial group theory comes about.  

I   discussed  Majumdar, Robbins and Zyskin's problem with Robert Kleinberg, to whom I am grateful   for  recognizing it as   RNA-folding and explaining  Algorithm~\ref{Alg: first} to me.  

Recently, Sergei Ivanov rekindled my interest in these issues, which he calls   \emph{precise area problems}.   
I thank him for discussions  and particularly for drawing my attention to the problem  of calculating area in $\langle  a, b \mid a^k, b^k;  \ k \in \N \rangle$.  Ivanov recognized what Jiang called the \emph{width} of an element of the free group $F(a,b)$ to be area in $\langle a, b \mid a^k, b^k;  \ k \in \N \rangle$.  He told me of the 1991 article \cite{GrK}     in which
 Grigorchuk and Kurchanov had  given an algorithm to compute width   and had asked whether it can be done in polynomial time.    Jiang's  motivation for defining width came from the problem of finding  the minimal number of fixed points in the homotopy class of a continuous self-map of a compact surface.      Section~\ref{lcds} includes some explanation of how this comes to be related to width.

In independent work  Ivanov has also   solved   Grigorchuk and Kurchanov's problem  \cite{Ivanov2}, also  using  Nussinov and   Jacobson's algorithm  as his starting point.  Indeed, he has taken the approach further.  He  gives non-deterministic log-space, linear-time algorithms (from which polynomial time deterministic solutions follow by  Savitch's theorem) calculating area  for a family of  presentations which includes $\langle  a, b \mid a, b  \rangle$ and   $\langle  a, b \mid a^k, b^k; \  k \in \N \rangle$.  He has also extended the techniques to $\langle  a, b \mid a^{-1}b^{-1} a b  \rangle$ and some other related presentations, and has derived  consequences for problems of computing the areas of discs spanned by polygonal curves in the plane.   
 
Finally, I am pleased to thank an anonymous referee for a thoughtful reading.

  \section{The   Grigorchuk--Kurchanov problem: area in $\langle  a, b \mid a^k, b^k; k \in \N \rangle$} \label{G-K}

We begin with some preliminaries concerning diagrams which display how a word which represents the identity in a group freely equals a product of conjugates of the defining relations.  The standard such diagram is known as a  
 van~Kampen diagram (see e.g.\ \cite{BRS}).  For us, it will useful to consider a variant.  We define a    \emph{cactus diagram for a word $w$ on $\set{a^{\pm1}, b^{\pm1}}$} to be a finite planar contractible 2-complex which has
\begin{itemize}
\item edges directed and labeled by $a$ or $b$ (as usual),
\item 2 types of faces: $a$-faces and $b$-faces, around whose perimeters we read a word on  $a^{\pm
1}$ or on $b^{\pm 1}$, respectively,
\item around the perimeter of the complex we read $w$,
\item no two $a$-faces have a common vertex; ditto $b$-faces.
\end{itemize}
 
Figure~\ref{cactus fig} shows an example of a van~Kampen diagram and a cactus diagram.   

The \emph{area} of a cactus diagram is the number of faces whose perimeter
word has non-zero exponent-sum.

\begin{figure}[ht]
\psfrag{a}{\footnotesize{$a$}}
\psfrag{b}{\footnotesize{$b$}}
 \centerline{\epsfig{file=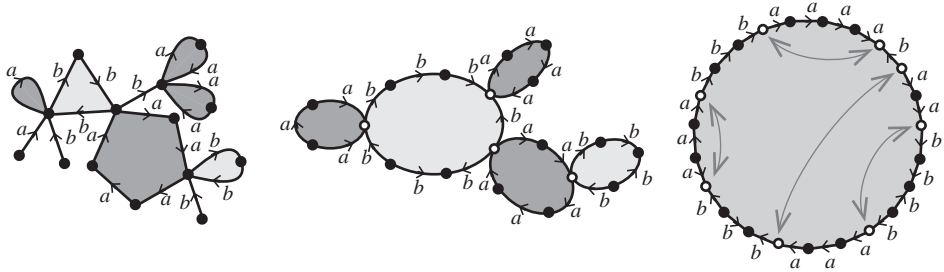}} 
 \caption{Left: A van Kampen diagram   over  $\langle  a, b \mid a^k, b^k; k \in \N \rangle$ for the word $a^{-1} a^2 b^3 a^4   b^{-1} a^2 b^2 b^{-1} b a^3 b b^{-1} b$.  Centre: a  cactus diagram for the same word.  Right: a polygon which gives the cactus diagram after identifying the  `pinch' vertices (the white vertices) as shown. }
  \label{cactus fig}
\end{figure}

A well known lemma of van~Kampen  tells us that $\Area(w)$, as defined in Section~\ref{intro}, is the minimal $N$ such that there is a van~Kampen diagram for $w$ with $N$ faces.  (Again, see  \cite{BRS} or  other surveys.)  Correspondingly:

\begin{lemma}
The area   $\Area(w)$ of  a word  $w$ on $\set{a^{\pm1}, b^{\pm1}}$, with respect to $\langle  a, b \mid a^k, b^k; k \in \N \rangle$   is   the   minimal $N$ such that there is a cactus diagram for $w$ of area $N$.  
\end{lemma}

\begin{proof}  This result follows from how a cactus diagram can be transformed   into a van~Kampen diagram over  $\langle  a, b \mid a^k, b^k; k \in \N \rangle$ and vice versa.  

Given a cactus diagram, fold together inverse pairs of edges around each face, leaving the perimeter word unchanged---if a face's perimeter word had zero exponent-sum, that face is replaced by  a 1-dimensional complex (in fact, a tree); otherwise, it is replaced by a 2-complex with a single 2-cell (and some 1-cells), and the perimeter word of that 2-cell is  a power of $a$ or $b$.   The result is a van~Kampen diagram whose area (that is, number of 2-cells) is the area of the cactus diagram, since the faces in the cactus diagram whose perimeter had  zero exponent-sum (those  we did not count in our definition of area of a cactus diagram)   collapse to trees and so do not contribute.

In the other direction, given a  van Kampen diagram obtain a cactus diagram of the same or lower area by a succession of moves: if there are two $a$-faces (or similarly two $b$-faces) with a common vertex (or indeed a larger common subcomplex), replace them by a single face in such a way as not to change the diagram's perimeter word; replace any edge in the diagram not in the boundary of a face by a bigon, again not changing the perimeter word. The former type of move decreases the number of 2-cells.  The latter adds 2-cells, but they have  perimeters of  zero exponent-sum, so do not contribute to area.  Therefore the resulting cactus diagram has area at most the area of the   van Kampen diagram.
\end{proof} 
 
The following lemma, which is illustrated in Figure~\ref{cactus fig},  will be crucial for us later in proving the correctness of our algorithms.   It follows essentially immediately from the fact that a cactus diagram is a tree-like arrangement of discs.  
 
\begin{lemma} \label{why cactus diagrams} 
Any cactus diagram $C$ for a word  $w$ on $\set{a^{\pm1}, b^{\pm1}}$, can be obtained as follows.   
Take a polygonal face whose edges are directed and labeled by $a$ and $b$ so that one reads $w$ around the perimeter.   Call the vertices where an $a$-edge meets a $b$-edge \emph{pinch vertices}.  
Pick any pinch vertex $v$.  Identify it with some other (suitably chosen) pinch
vertex so as to convert that face into two faces wedged at a point.
Pick any pinch vertex on either of the new faces, and likewise
identify it with another (suitably chosen)  pinch vertex on the same face.  Repeat until
no pinch vertices remain.   (The conditions on the choices of vertices in the statement of this lemma are crucial.  At each step a pair of pinch vertices is identified.  The first  of these pinch vertices can be chosen  arbitrarily among all pinch vertices on all faces, but then its mate is determined by $C$.)   
\end{lemma}

It is convenient to have one further interpretation of $\Area(w)$, via null-sequences like that  we described for $\langle a, b \mid a, b \rangle$ in  Section~\ref{background}.  Here is how this adapts  to $\langle  a, b \mid a^k, b^k; k \in \N \rangle$.  A word $w$ on $a^{\pm 1}, b^{\pm 1}$ can be converted to the empty word by a sequence of two types of move: (1)  delete a subword $a^k$ or $b^k$ where $k \in \Z$, (2)  cancel an inverse pair of adjacent letters $a^{\pm 1} a^{\mp 1}$ or $b^{\pm 1} b^{\mp 1}$.  Then $\Area(w)$ in  $\langle  a, b \mid a^k, b^k; k \in \N \rangle$  is the minimal $N$ such that there is such a \emph{null-sequence} that employs $N$ moves of the first type. 
 
Now consider a word
$w = a b^{i_1} a b^{i_2} \ldots a b^{i_k}$
where $i_1, \ldots, i_k$ are non-zero integers.  Clearly, $\Area(w) \leq
k+1$: delete the subwords $b^{i_j}$ one at a time and then delete all the
$a$ together.  (An associated cactus diagram has a singe $a$-face with perimeter $a^k$ and has one $b$-face attached for each   $b^{i_j}$.)

Recall that Theorem~\ref{comparison} asserts that  $\Area(w) \leq k$ if and only if there are $j_1 < \cdots <  j_l$
such that $i_{j_1} + \cdots + i_{j_l} = 0$.

\begin{proof}[Proof of Theorem~\ref{comparison}]

For the `if' direction, delete the $b^{i_j}$ for which $j \notin \set{j_1,
\ldots, j_l}$ one at a time, then delete the $l-1$ powers of $a$ that sit
between the remaining $b$ letters, then cancel away the remaining $b$  at no
cost, then delete the remaining $a$ letters.   The total cost  is $(k-l)
+ (l-1) +1 =k$.   

For the `only if' direction, notice first that every cactus diagram for $w$ has at least $k+1$ faces as can be seen as follows.  Consider the process of constructing a cactus diagram described in Lemma~\ref{why cactus diagrams}.  As it progresses each face has perimeter labelled (up to cyclic permutation) either by a word on $a, a^{-1}$, or by  a word on $b, b^{-1}$, or by $u_1 v_1 \cdots  u_j v_j$ for some non-empty words $u_i$ on $a, a^{-1}$ and $v_i$ on $b, b^{-1}$.  Say a face has syllable length $1$ or $2j$, accordingly.      
Consider the sum $S$ of the syllable lengths of the faces which do not have syllable length $1$.  Each pinch increases the number of faces by $1$ and either leaves $S$ unchanged or decreases it by $2$.   We arrive at our cactus diagram when $S$ reaches $0$, so every cactus diagram for $w$ indeed has at least $k+1$ faces.

The word around an $a$-face of a cactus diagram for $w$ is a power of $a$.  So a cactus diagram exhibiting $\Area(w) \leq k$ can have no more than $(k-1)$ $b$-faces that contribute to its area and so must have one that is labeled by a 
word $b^{i_{j_1}}   \cdots  b^{i_{j_l}}$ that is made by concatenating some of the $b^{i_j}$ subwords from $w$ and which  has     exponent sum zero.
\end{proof} 

\begin{proof}[Proof of case \eqref{second} of Theorem~\ref{main}]
Now we give our algorithm for area in $\langle a, b \mid a^{k},  b^{k}; \  k \in \N  \rangle$.  
In essence, we  combine    Algorithms~\ref{Alg: first} and \ref{Alg: subset sum}.   As stated, our algorithm  only finds the areas of words   $w$ of the form $a^{i_1} b^{i'_1} .... a^{i_m} b^{i'_m}$ where  $i_1, i'_1, ..., i_m, i'_m$ 
are non-zero integers.  But this represents  no   loss in generality as $\Area(a^{i_1}) = \Area(b^{i'_1}) =1$  for all $i_1, i'_1 \neq 1$, and replacing a word by a cyclic conjugate and freely reducing it does not change its area.

For  $1 \leq j \leq k \leq m$ and for   $r \in \Z$ define
\begin{align*}
w_{j,k} & \  := \  a^{i_j} b^{i'_j} .... a^{i_k} b^{i'_k}, \\ 
w_{r;j,k} & \ :=  \ a^r b^{i'_j} .... a^{i_k} b^{i'_k}, \\
w_{j,k;r}  & \  := \  a^{i_j} b^{i'_j} .... a^{i_k} b^r.
\end{align*}
So  $w_{r;j,k}$ and $w_{j,k;r}$   are $w_{j,k}$ with the first and last `syllable' (respectively) replaced by $a^r$ and  $b^r$ (respectively).  

Let $n =   |i_1| +  |i'_1| +  ... + | i_m| +  |i'_m|$, the length of $w$.  
For triples of integers $j,k,r$ such that $1 \leq j \leq k \leq
m$ and  $|r| \leq n$ our algorithm will compute two arrays of integers $A_{j,k;r}$ and $A_{r;j,k}$, which will be the 
areas of $w_{j,k;r}$ and $w_{r;j,k}$, respectively, for reasons we will explain.  The  computation of $A_{j,k;r}$ and $A_{r;j,k}$ will proceed in increasing order of $|k-j|$. 
There are $\tilde{O}(n^3)$ such triples, $j,k,r$.
 Our algorithm will output $A_{i_1; 1,m}$, which will be the area of $w$.

 \begin{algorithm}[ht]
    \caption{--- Area in $\langle a, b \mid a^{k},  b^{k}; \  k \in \N  \rangle$\newline 
  $\circ$ \  Input  a word $w=a^{i_1} b^{i'_1} .... a^{i_m} b^{i'_m}$  where  $i_1, i'_1, ..., i_m, i'_m$ are   non-zero integers.
    \newline  
  $\circ$ \  Return $\Area(w)$   in time  $O(n^4)$, where $n =  |i_1| +  |i'_1| +  ... + | i_m| +  |i'_m|$.   }
    \label{Alg: second}
    \begin{algorithmic}
    		\State For   $j, k, r$, outside the range ($1 \leq j \leq k \leq
m$ and  $|r| \leq n$) define $A_{j,k;r} := A_{r;j,k} := \infty$
     		\State For $1 \leq j   \leq
m$ and  $|r| \leq n$,  define  $A_{j,j;r}$ and $A_{r;j,j}$ to be $1$  if $r=0$ and to be $2$  otherwise 
		\State For $s=1$ to $m-1$
		\State \qquad  For $j= 1$ to $m-s$ 
		\State \qquad \qquad For $r= -n$ to $n$ 
		\State \qquad \qquad \qquad Define $k:= j+s$  		
		\State  \qquad \qquad  \qquad    Define $A_{j,k;r}$ to be  the minimum of  
		\State  \qquad \qquad \qquad \qquad $\set{ \left. A_{i_j; j,l} + A_{l+1,k;r} \right|   j \leq l  < k } \cup  
		\set{ \left. A_{i_j+i_l; j, l-1} + A_{l+1, k; i'_l +r}  \right|  j \leq l \leq k  }$
		\State    \qquad \qquad  \qquad  Define $A_{r; j,k}$ to be  the minumum of  
		\State  \qquad \qquad \qquad \qquad  $\set{ \left. A_{r;j,l} + A_{l+1,k;i'_k} \right|   j \leq l  < k } 
		\cup \set{ \left. A_{r+i_l; j, l-1}  + A_{l+1, k; i'_k +i'_l}   \right|  j \leq l \leq k  }$
		\State Return $A_{i_1; 1,m}$
  \end{algorithmic}
  \end{algorithm} 

Here is why $A_{j,k;r} = \Area(w_{j,k;r})$ and $A_{r;j,k} = \Area(w_{r;j,k})$.
Consider a polygonal face $f$ with its edges directed and labeled so that around the perimeter we read $w_{j,k;r}$ anticlockwise (say) starting from a pinch vertex $v$.  Lemma~\ref{why cactus diagrams} tells us that in any cactus diagram for $w_{j,k;r}$ (in particular,   one of  minimal area),  $v$ is identified with some
other pinch vertex $u$ on $f$.  This identification  subdivides $f$ into
two faces $f_1$ and $f_2$.  Reading anticlockwise around the boundary of $f$, the vertex $u$ is
either preceded by $a^{\pm 1}$ and followed by $b^{\pm 1}$, or vice versa.
   The forms of the two words around $f_1$ and $f_2$ (read
anticlockwise from the common vertex)  differ accordingly.  In the
first case they are $w_{j,l}$ and $w_{l+1,k;r}$ for some $j \leq l  < k$.  In
the second case they are
$$a^{i_j} b^{i'_j} \cdots a^{i_{l-1}} b^{i'_{l-1}} a^{i_l}  \  \text{  and } \  b^{i'_l} a^{i_{l+1}} b^{i'_{l+1}} \cdots a^{i_k} b^r$$
for some  $j \leq l \leq k$.   These latter two words are cyclic permutations of
$$w_{i_j+i_l; j, l-1} =  a^{i_j + i_l} b^{i'_j} \cdots a^{i_{l-1}} b^{i'_{l-1}} \  \text{   and } \ 
w_{l+1, k; i'_1 +r} =  a^{i_{l+1}} b^{i'_{l+1}} \cdots a^{i_k} b^{i'_l+ r},$$
respectively.

Likewise, if we read $w_{r;j,k}$ around $f$, then  the words around $f_1$ and $f_2$ are  either $w_{r;j,l}$ and $w_{l+1,k}$, or    $$a^{r} b^{i'_j} \cdots a^{i_{l-1}} b^{i'_{l-1}} a^{i_l}  \  \text{  and } \  b^{i'_l} a^{i_{l+1}} b^{i'_{l+1}} \cdots a^{i_k} b^{i'_k}.$$  The latter pair are    cyclic permutations of
$$w_{r+i_l; j, l-1} =  a^{r + i_l} b^{i'_j} \cdots a^{i_{l-1}} b^{i'_{l-1}} \  \text{   and } \ 
w_{l+1, k; i'_k +i'_l} =  a^{i_{l+1}} b^{i'_{l+1}} \cdots a^{i_k} b^{i'_k + i'_l},$$
respectively.

In any such sequence of pinches that creates a
cactus diagram for $w$, the words $w_{j,k;r}$ and  $w_{r; j,k}$ that arise
around the faces have $|r| \leq n$.

Here is why the algorithm halts in time $\tilde{O}(n^4)$.  As already noted, the arrays $A_{j,k;r}$ and  $A_{r;j,k}$ have size $O(n^3)$.  Computing each  $A_{j,k;r}$ and  $A_{r;j,k}$  involves calculating the minimum of $O(n)$ sums of pairs of prior computed entries.  
\end{proof}

In the light of the advances on subset sum in \cite{KX} and on RNA-folding in \cite{BGSW}, it seems likely this  
$\tilde{O}(n^4)$ bound   could be improved.  

\section{Liquid crystal design, and counting regular values   and   fixed points} \label{lcds}

Here is a sketch of the  unlikely association between combinatorial group theory and liquid crystal design found by Majumdar, Robbins and Zyskin (MRZ) \cite{Robbins1}.  

In a liquid crystal display a rectangular block $P$ of `nematic liquid crystal' is sandwiched between two polarizing filters which are offset 90 degrees from each other.  Light passes through the first filter, then through the   liquid crystal, then meets the second filter.   What then happens  depends on the liquid crystal.  The word `nematic' is derived from the Ancient Greek word for `thread.' The molecules in  a nematic liquid crystal are long and straight and line up next to each other,  so are naturally modeled by a continuous unit-vector field $n : P \to S^2$.     (Actually, the molecules lack a  preferred orientation, so a director field $n : P \to \mathbb{R} \textup{P}^2$ may be more appropriate, but such a field can be oriented in straightforward settings.)  The faces of $P$   are  coated with an `alignment layer' which forces the molecules there to line up tangent  to the faces---that is, it imposes a tangent boundary condition on $n$.  The alignment of the  face incident with the first filter and the  face incident with the second are set at 90 degrees to each other, leading  the molecules to arrange themselves in a helical manner twisting 90 degrees through the block.  So arranged, the  molecules  rotate  the polarization of the light 90 degrees, and  it shines through the second filter.   An electric field can be applied  to the two faces of $P$ incident with the filters so as to reconfigure the molecules so they align perpendicular to the two filters.  They then leave the  polarization   unchanged, and so   no  light  emerges.
  
Looking to build on this, researchers have explored  the optic properties of nematic liquid crystals of a variety of shapes and with different  tangent boundary conditions.    The Dirichlet energy  $E(n)  = \int_P  \abs{ \nabla n }^2 dV$   of $n$  measures how variable $n$ is.   The molecules in a nematic liquid crystal arrange themselves so as to  minimize $E(n)$ locally (that is, so that $n$ is harmonic).  This local `arranging' can be looked at as a homotopy, so  determining the infimum of $E(n)$ within a homotopy classes of vector fields  $n : P \to S^2$ satisfying the given boundary tangency conditions is a step towards understanding the  optic properties.    
   
A special  case where there have been significant results on this problem   is when $$P  \ = \  \set{ (x_1,x_2,x_3) \in \R^3 \mid 0 \leq x_i \leq L_i}$$ is a rectangular block of side lengths $L_1$, $L_2$ and $L_3$, which we will assume for convenience are all strictly greater than $1$, and the homotopy class $h$ contains a representative which  is invariant on reflection through each of the three planes $x_i = L_i/2$.  A sphere of radius $1$ and centered at the origin intersects $P$ in a spherical triangle $O$.  Restricting $n$ to $O$ defines a continuous unit vector field $O \to S^2$, and the tangent boundary condition  implies that   $n$ maps points $p$ on   each side $\sigma$  of $O$ to the great circle  of $S^2$ that contains $\sigma$.    (The tangency condition says that the unit tangent vector $n(p)$ at $p$ lies in the face of $P$ containing $\sigma$.  Translating $n(p)$ to begin at the origin, it remains in the plane containing $\sigma$ and its end point is on $S^2$.)

Let $C_T(O, S^2)$ denote the space of all such $O \to S^2$.     MRZ explain that  the $h$ as above are in one-to-one correspondence with the homotopy classes $H$ of $C_T(O, S^2)$.   Moreover, they show that the infimal   Dirichlet energy over $h$ is bounded from above and below by    the infimal   Dirichlet energy over $H$ times suitable   constants.  
This then motivates a search for estimates on the  infimal   Dirichlet energy for homotopy classes $H$ of $C_T(O, S^2)$.

There is a classification of the homotopy classes of $C_T(O, S^2)$, which leads to the following. 
The 2-sphere subdivides into eight spherical triangles, one for each octant of $\mathbb{R}^3$.  
For $v  \in C_T(O, S^2)$ and  regular  values $s_0, s_1, s_2, s_3$ of $v$ (that is, values where $\textup{det} (\nabla v) \neq 0$) in a certain four of these eight spherical triangles, MRZ give  an estimate on $E(v)$ of a constant times  the numbers of preimages  of the $s_0$, $s_1$, $s_2$, and $s_3$.  They then look for a  $\widetilde{v}$  homotopic to $v$   in $C_T(O, S^2)$ that improves this estimate.  Preimages of $s_i$ have a sign according to whether the determinant of $\nabla v$ is positive or negative.  The aim is to change $v$ so as to cancel pairs of preimages of opposite sign.   When this is done optimally, MRZ's estimate then gives the infimal   Dirichlet energy   for the homotopy classes.   

The subtlety is that  such cancellations cannot freely be achieved.  In place of $v: O \to S^2$, consider a map $\phi: D^2 \to  S^2$ from the 2-disc to a 2-sphere. Suppose  $s_0, s_1, s_2, s_3$ are regular values of $\phi$.  View $\phi$ as  a  null-homotopy  of based loops carrying $\phi \restricted{\partial D^2}$ to the constant loop.  
MRZ's estimate for  Dirichlet energy corresponds to counting how many times the loop  crosses      $s_0, s_1, s_2, s_3$ during the null-homotopy.  And the optimization problem corresponds to changing the null-homotopy of   $\phi \restricted{\partial D^2}$ away from a   disc-neighbourhood of $s_0$.  The complement of    that   disc-neighbourhood  in $S^2$ is  $D^2$.  If we regard   $s_1, s_2, s_3$ as punctures, then $\phi \restricted{\partial D^2}$ represents an element of $\pi_1(D^2 \ssm \set{s_1, s_2, s_3})$.  This group is the rank-$3$ free group $F_3$ and a basis $F(a,b,c) = F_3$ can be chosen so that each time a puncture is crossed, the reduced word representing the group element changes by inserting or removing an $a$, $b$ or $c$.  The optimization problem then amounts to taking the reduced word $w$ representing $\phi \restricted{\partial D^2}$ and reducing it to the empty word by  removing an  $a$, $b$ or $c$ (inserting one is always superfluous) and then freely reducing, as few times as possible---in other words, finding the spelling length of $w$.

In the setting of $C_T(O, S^2)$, MRZ  fully describe the $w$ that arise and establish lower bounds on their spelling lengths.  But they remark that the story remains incomplete:  their bounds on spelling length are ad hoc and  they say that they ``are not aware of general results for obtaining the minimum spelling length over a product of conjugacy classes.'' 
  
The way \emph{width} arises in the work of Jiang \cite{Jiang} on minimizing fixed points within homotopy classes
 is harder to p\'recis.  Jiang considers  $f: M \to M$ where $M$ is a connected compact surface $M$ with trivial $\pi_2$.  He characterizes when there exists $g$ homotopic to $f$ with  $k$ fixed points points of indices $i_1,  \ldots, i_k$.  To this end he argues that we can assume the fixed points of $g$ are all in the interior of a certain disc $D \subset M$ and he considers the map  $M \ssm \text{int} \, D \to M \times M \ssm \Delta$, given by $x \mapsto (x,g(x))$, where $\Delta$ denotes the diagonal of $M \times M$.   This induces a map  $\phi: \pi_1 (M \ssm  \text{int} \,  D ) \to \pi_1( M \times M \ssm \Delta)$, the target being the group of pure 2-braids in $M$.   Jiang's characterization  is a number of conditions on $\phi$ including that it maps the group element represented by $\partial D$ to 
 \begin{equation} 
 v_1 B^{i_1}v_1^{-1} \cdots v_k B^{i_k}v_k^{-1} \label{prodconj}
\end{equation} 
  for certain $v_i$ and a certain braid $B$.    
 He then reinterprets these conditions as a collection of equations in the kernel $K$ of the map  $ M \times M \ssm \Delta \to M  \times M$, which is a free group with free basis a certain family of conjugates of $B$.   On account of \eqref{prodconj} appearing in these equations, Jiang's characterization then yields that the minimal $k$ such that there is $g$ that is  homotopic to $f$ and has $k$   fixed points is the minimum among the widths in $K$ of a certain family of words.

\bibliographystyle{alpha}
\bibliography{$HOME/Dropbox/Bibliographies/bibli}

\ni  \textsc{Timothy R.\ Riley} \rule{0mm}{6mm} \\
Department of Mathematics, 310 Malott Hall,  Cornell University, Ithaca, NY 14853, USA \\ \texttt{tim.riley@math.cornell.edu}, \
\href{http://www.math.cornell.edu/~riley/}{http://www.math.cornell.edu/$\sim$riley/}

 \end{document}